\documentclass[10pt]{article}

\usepackage{amssymb,amsthm,amsmath,hyperref}

\numberwithin{equation}{section}

\newtheorem{thm}{Theorem}[section]
\newtheorem{lem}{Lemma}[section]

\theoremstyle{definition}

\theoremstyle{remark}
\newtheorem{rem}{Remark}[section]

\allowdisplaybreaks

\begin{document}

\title{An elementary proof of the global existence and uniqueness theorem to 2-D incompressible non-resistive MHD system}

\author{  Ting Zhang \\
  Department of Mathematics, Zhejiang University,
Hangzhou 310027, China}
\date{}
\maketitle

\begin{abstract}
 In this paper, we provide a much simplified proof of the main result
in \cite{Lin12-2} concerning the global existence and uniqueness of smooth
solutions to the Cauchy problem for a 2D incompressible viscous and non-resistive MHD system under the assumption that the initial data are close to some
equilibrium states. Beside the classical energy method, the interpolating
inequalities and the algebraic structure of the equations coming from the
incompressibility of the fluid are crucial in our arguments. We
combine the energy estimates with the $L^\infty$ estimates for time
slices to deduce the key $L^1$ in time estimates. The latter is
responsible for the global in time existence.

\textbf{Keywords}: MHD system; existence; uniqueness.

\textbf{2010 AMS Subject Classification}: 35Q35, 76W05.
\end{abstract}

\section{Introduction.}
In this paper, we consider the global existence of strong solutions to the following 2D incompressible viscous and non-resistive magnetohydrodynamics(MHD) system \cite{Lin12-3,Lin12-2},
   \begin{equation}
  \left\{
  \begin{array}{l}
    \partial_t \phi +v\cdot\nabla \phi=0,\ \ (t,x)\in\mathbb{R}^+\times\mathbb{R}^2,\\
        \partial_t v+v\cdot\nabla v-\Delta v+\nabla p=-\mathrm{div}[\nabla \phi\otimes\nabla \phi],\\
            \mathrm{div}v=0,\\
            (\phi,v)|_{t=0}=(\phi_0,v_0),
  \end{array}
  \right.\label{m0-E1.1}
\end{equation}
where the initial data $(\phi_0,v_0)$ close enough to the equilibrium state $(x_2,0)$.
 Here $\phi$,  $v=(v_1,v_2)^\top$ and $p$
denote the magnetic potential, velocity and scalar pressure of the fluid respectively.

The system (\ref{m0-E1.1}) is formally equivalent to the following 2D MHD system
    \begin{equation}
      \left\{
      \begin{array}{l}
        \partial_t b+v\cdot\nabla b =b\cdot\nabla v,\ \ (t,x)\in\mathbb{R}^+\times\mathbb{R}^2,\\
        \partial_t v+v\cdot\nabla v-\nu \Delta v+\nabla p=-\frac{1}{2}\nabla(|b|^2)+b\cdot\nabla b,\\
            \mathrm{div}v=\mathrm{div}b=0,\\
               ( b,v)|_{t=0}=(b_0, v_0),
      \end{array}
      \right.\label{m0-E1.2}
    \end{equation}
where $b=(b_1,b_2)^{\top}$,   $v=(v_1,v_2)^{\top}$ and $p$
denote the magnetic field, velocity and scalar pressure of the fluid respectively. In face, the condition $\mathrm{div}b=0$ implies the existence of a scalar function $\phi$ such that $b=(\partial_2\phi,-\partial_1\phi)^\top$, and the system (\ref{m0-E1.2}) becomes  the (\ref{m0-E1.1}). Magnetohydrodynamics (MHD) is the study of the dynamics of electrically conducting fluids, such as plasmas, liquid metals, and salt water or electrolytes. The MHD system describes   many phenomena such as the geomagnetic dynamo in geophysics and solar winds and solar flares in astrophysics \cite{Alfven,Biskamp,Davidson,Priest}.

There are many interesting results on the global regularity problem for the 2D MHD system with dissipation.  The viscous and resistive MHD system has a unique global classical solution with the initial data $(v_0,b_0)\in H^m(\mathbb{R}^2)$, $m\geq2$ (e.g. \cite{Duvaut,Sermange}). The inviscid and resistive MHD system
 has a global weak solution with $(v_0,b_0)\in H^1(\mathbb{R}^2)$ (e.g. \cite{Cao}). With mixed partial dissipation and additional (artificial) magnetic diffusion in the 2D incompressible MHD system, Cao and Wu \cite{Cao} proved its global wellposedness results for any initial data in $H^2(\mathbb{R}^2)$.
Considering the ideal MHD system (i.e. inviscid and non-resistive), Bardos, Sulem and Sulem \cite{Bardos} proved the global existence of the classical solution when the initial data $(b_0,v_0)$ close to  the equilibrium state $(B_0,0)$. It is point out in \cite{Bardos} that the fluctuations $v+b-B_0$ and $v-b+B_0$ propagate along the $B_0$ magnetic field in opposite directions. Then, a strong enough magnetic field will reduce the nonlinear interaction and inhibit formation of strong gradients \cite{Bardos,Frisch,Kraichnan}.
Under some special conditions on the initial data,  F.H. Lin, L. Xu and P. Zhang \cite{Lin12-2} can transfer the system (\ref{m0-E1.1}) to the  2D viscoelastic fluid system.
Using the Lagrangian transformation method and the anisotropic Littlewood-Paley analysis techniques,   they \cite{Lin12-2} obtained  the global wellposedness result. The arguments involved, despite its
general interests, were rather complicated. The aim of this note is to
give a new and simple proof, which involves only the energy estimate
method, interpolating inequalities and couple elementary observations.

After substituting $(\phi,v)=(x_2+\psi, v)$ into (\ref{m0-E1.1}), one obtains the following system for $(\psi,v)$:
 \begin{equation}
  \left\{
  \begin{array}{l}
    \partial_t \psi +v\cdot\nabla \psi+  v_2=0,\ \ (t,x)\in\mathbb{R}^+\times\mathbb{R}^2,\\
        \partial_t v_1+v\cdot\nabla v_1 -\Delta v_1+\partial_1 p+ \partial_1\partial_2\psi=
        -\mathrm{div}[\partial_1 \psi \nabla \psi],\\
        \partial_t v_2+v\cdot\nabla v_2 -\Delta v_2+\partial_2 p+ (\Delta+\partial_2^2)\psi=
        -\mathrm{div}[\partial_2 \psi \nabla \psi],\\
            \mathrm{div}v=0,\\
            (\psi,v)|_{t=0}=(\psi_0,v_0).
  \end{array}
  \right.\label{m0-E1.5-N2}
\end{equation}
Here and in what follows, we denote that   $\partial_i=\partial_{x_i}$ with $i\in\{1,2\}$.

It is easy to obtain the local existence and uniqueness of the solution for the system (\ref{m0-E1.5-N2}), see Theorem \ref{m0-Thm3.1}. In this paper, we will prove the following global existence and uniqueness of the solution of the  system (\ref{m0-E1.5-N2}) with the small and smooth initial data. Assume that the initial data satisfy
    \begin{equation}
     \nabla \psi_0,v_0\in
         H^2(\mathbb{R}^2),
       \ \mathrm{div} v_0=0, \label{m0-E1.3}
    \end{equation}
    and
    \begin{equation}
        e^{- {|\xi|^2t} }\widehat{\nabla\psi}_0(\xi),
      \ e^{- {|\xi|^2t} }\widehat{v}_{0}(\xi)\in{L^2([0,\infty);L^1_\xi)} ,\label{m0-E1.3-09}
    \end{equation}
    where $\widehat{f}$ is the Fourier transform of the function $f$.
Let
    $
    A_{0}=A_{1,0}+A_{2,0}
    $ with
        $$
        A_{1,0}= \|\nabla \psi_0\|_{H^2}+
                  \|v_0\|_{H^2},\
    A_{2,0}=
       \| e^{- {|\xi|^2t} }\widehat{\nabla\psi}_0\|_{L^2([0,\infty);L^1_\xi)}+
      \|e^{- {|\xi|^2t}}\widehat{v}_{0}\|_{L^2([0,\infty);L^1_\xi)},
    $$
and  $ A_{T}=A_{1,T}+A_{2,T}$ with
$$
        A_{1,T}= \|v\|_{L^\infty([0,T]; H^2 )}+\|\nabla \psi\|_{L^\infty([0,T];H^2)}
      +\|\nabla v\|_{L^2([0,T]; H^2 )}+\|\partial_1\nabla \psi\|_{L^2([0,T]; H^1 )},
        $$
    $$
    A_{2,T}=
     \|\widehat{v} (\xi)\|_{L^2([0,T]; L^1_\xi )}+ \|\widehat{\partial_1  \psi}(\xi)\|_{L^2([0,T];L^1_\xi )},
    $$
Denote
 $$
 E^n_T= \left\{ (\psi,v,\nabla p)\left| \begin{array}{l}
    \nabla \psi,v, \in C([0,T];H^n),\nabla p\in C([0,T]; H^{n- 1}), \\
     (\partial_1\nabla \psi, \nabla v)
   \in L^2([0,T]; H^{n- 1} \times  H^n ),\ \widehat{v} (\xi),\ \widehat{\partial_1  \psi}(\xi)\in{L^2([0,T]; L^1_\xi  )},
   \end{array}
   \right.
   \right\}.
    $$
We use the notation $E^n$ if $T=\infty$ by changing the time interval $[0,T]$ into $[0,\infty)$ in the above definition.
\begin{thm}\label{m0-Thm1.1}
Assume that the initial data $(\psi_0,v_0)$ satisfy (\ref{m0-E1.3})-(\ref{m0-E1.3-09}),  then there exists
    a positive constant $c_0$ such that if
        \begin{equation}
         A_0\leq c_0,\label{m0-E1.8}
        \end{equation}
    then the system (\ref{m0-E1.5-N2}) has a unique global solution $(\psi,v,\nabla p)\in E^2$ satisfying
                 \begin{eqnarray}
      A_T
      &\leq& CA_{0},\label{m0-E1.12}
    \end{eqnarray}
and
    \begin{equation}
      \|\nabla p\|_{L^\infty([0,T]; H^1 )}
       \leq  C(1+
      c_0)A_0,\label{m0-E1.13}
    \end{equation}
    for all $T>0$,
where $C$ is a positive constant independent of $T$.
\end{thm}
\begin{rem}
  One can see that if $\nabla\psi_0,v_0\in \dot{B}^{0}_{2,1}$, then (\ref{m0-E1.3-09}) holds. Then $A_{2,0}$ can be replaced by
    $A_{2,0}=\|\nabla\psi_0\|_{\dot{B}^{0}_{2,1}}+\|v_0\|_{\dot{B}^{0}_{2,1}}$.
\end{rem}
\begin{rem}
Under the assumptions of Theorem \ref{m0-Thm1.1}, if $(\nabla \psi_0,v_0)\in H^n(\mathbb{R}^2)\times H^n(\mathbb{R}^2)$, $n\geq3$, then we can easily obtain that $(\psi,v,\nabla p)\in E^n$
    and omit the details.
\end{rem}

\begin{rem}\label{m0-Rem1.2}
In \cite{Lin12}, using the anisotropic Littlewood-Paley analysis,  H.F. Lin and P. Zhang proved a global wellposedness result of
a 3D toy model. In \cite{Lin14}, we provided a new and simple proof. The proof of this paper is similar to that in  \cite{Lin14}.
The three key technical points in our proof are:
\begin{description}
  \item[(1)] interpolating estimates, like Lemma  \ref{m0-L3.2};
  \item[(2)] using algebraic structure: $\mathrm{div} v =0$ to inter-changing the derivative estimates for  $\partial_1$ and
$\partial_2$;
  \item[(3)]  using the first equation of (\ref{m0-E1.5-N2}) to reduce
"$L^1$ in time estimates" (\cite{Lin12-2}) which is the key to the global
existence result to "energy estimates and
$L^\infty$ estimates for time slices" that are relatively easy to obtain.
\end{description}
 In fact, the basic strategy for the proofs is rather clear. Using the
basic energy laws, one reduces the problems to estimating certain nonlinear
terms of particular forms.
For example,   one of the difficulties of the proofs would be to control
the following type term,
    \begin{equation}
      \int^T_0\int_{\mathbb{R}^2} \partial_2 v_2  (\partial_2^3\psi)^2 dxdt.
    \end{equation}
Since the first derivative of $\psi$, $\partial_1\psi$, decay faster than $\partial_2\psi$ (by energy laws),
in \cite{Lin12-2}, H.F. Lin, L. Xu and P. Zhang explored
such anisotropic behavior by using the
anisotropic Littlewood-Paley theory to conclude the key estimate that
 $v_2\in L^1(\mathbb{R}^+;Lip(\mathbb{R}^2))$.
In this paper, we will show that   $\nabla \psi\in L^4_T(L^{ \infty})$ in Lemma \ref{m0-L3.2} by the interpolating estimate. Then, we use (\ref{m0-E1.5-N2})$_1$ twice  to obtain that
    \begin{eqnarray}
     &&\left| \int^T_0\int_{\mathbb{R}^2} \partial_2 v_2 (\partial_2^3\psi)^2 dxdt\right|\nonumber\\
        &=&\left|\int^T_0\int_{\mathbb{R}^2}
         \partial_2(\partial_t\psi+v\cdot\nabla\psi  )(\partial_2^3\psi)^2 dxdt\right|\nonumber\\
        &\leq&\left|\int_{\mathbb{R}^2}
         \partial_2 \psi(\partial_2^3\psi)^2 dx\big|^T_0\right|+\left|\int^T_0\int_{\mathbb{R}^3}
        \partial_2 v_2\partial_2\psi(\partial_2^3\psi)^2 dxdt\right|+\ldots\nonumber\\
           &=&\ldots+\left|\int^T_0\int_{\mathbb{R}^2}
        \partial_2 (\partial_t\psi+v\cdot\nabla\psi ) \partial_2\psi(\partial_2^3\psi)^2 dxdt\right|+\ldots\nonumber\\
        &\leq&\ldots+ C\|\nabla \psi\|_{L^4_T(L^\infty(\mathbb{R}^2))}^2\|\nabla v\|_{L^2_T(L^2(\mathbb{R}^2))}\|\nabla \psi\|_{L^\infty_T(H^2(\mathbb{R}^2))}^2+\ldots,
    \end{eqnarray}
Refer the details in  Lemma \ref{m0-L3.6}.
This is a simple method for the issue concerning the anisotropic dissipative system similar to (\ref{m0-E1.5-N2}).
\end{rem}

\begin{rem}\label{m0-Rem1.3}
In Lemma  \ref{m0-L3.6},  we want to estimate the following term,
   \begin{eqnarray}
 \left|\int^T_0\int_{\mathbb{R}^2}
        \partial_2 v_1\partial_1\psi(\partial_2^3\psi)^2 dxdt\right|. \nonumber
    \end{eqnarray}
 Since $\partial_2^3\psi\in L^\infty_T(L^2)$ and
    $$
    \left|\int^T_0\int_{\mathbb{R}^2}
        \partial_2 v_1\partial_1\psi(\partial_2^3\psi)^2 dxdt\right|\leq \|\nabla v\|_{L^2_T(L^\infty(\mathbb{R}^2))}
        \|\partial_1 \psi\|_{L^2_T(L^\infty(\mathbb{R}^2))}\|\partial_2^3\psi\|_{L^\infty_T(L^2(\mathbb{R}^2))}^2,
    $$
     we need to estimate the term $\|\partial_1 \psi\|_{L^2_T(L^\infty(\mathbb{R}^2))}$. One cannot bound it if only have
     $\nabla \partial_1\psi\in L^2_T(H^1(\mathbb{R}^2))$. So we add the assumption (\ref{m0-E1.3-09}) on the initial data, and will get
        $$
       \|\partial_1 \psi\|_{L^2_T(L^\infty)}\leq C\|\widehat{\partial_1 \psi}(t, \xi)\|_{L^2_T(L^1_\xi)}\leq CA_{2,0}+CA_{1,T}+CA_T^2.
        $$
        Please see the details in Section \ref{m0-S4}.
\end{rem}

\begin{rem}
  Recently, using the anisotropic Littlewood-Paley analysis techniques,  J.H. Wu, Z.Y. Xiang and Z.F. Zhang \cite{Wu} obtain the global existence and decay estimates of global
strong solution when $\|(\nabla\psi_0,v_0)\|_{H^8\cap H^{-s,-s}\cap H^{-s,8}}\ll 1$, $s=\frac{1}{2}-\epsilon$. Under the special assumptions on the initial data, by the anisotropic Littlewood-Paley analysis techniques,  X.P. Hu and F.H. Lin \cite{Hu} also obtain the   existence of global
strong solution when the initial data are in the critical space.
\end{rem}

Let us complete this section by the notation we shall use in this paper.

\textbf{Notation.}   We shall denote by $(a|b)$ the $L^2$ inner product of $a$ and $b$,  $(a|b)_{\dot{H}^s}=\sum_{|\alpha|=s}(\partial^\alpha a|\partial^\alpha b)$, and $(a|b)_{H^s}=\sum_{k=0}^s(a|b)_{\dot{H}^k}$,   $C_T(X)=C([0,T];X)$ and $L^p_T(X)=L^p([0,T];X)$.

\section{A priori estimates for $A_{1,T}$}\label{m0-S3}
In this section, we prove some \textit{a priori estimates} for $A_{1,T}$ which are crucial for the global existence of the strong solutions for the 2D MHD system (\ref{m0-E1.5-N2}). We begin with the following
Gagliardo-Nirenberg-Sobolev type estimate, see \cite{Nirenberg}.

\begin{lem}\label{m0-L3.2}
  If the function $\psi$ satisfies that $\nabla \psi\in L^\infty_T(H^2(\mathbb{R}^2))$ and $\partial_1\nabla \psi\in L^2_T(H^1(\mathbb{R}^2))$,  then there holds
  \begin{eqnarray}
      \|\nabla \psi\|_{L^4_T(L^\infty(\mathbb{R}^2))}   &\leq& C      \| \partial_1\nabla \psi\|_{L^2_T(H^1(\mathbb{R}^2))}^\frac{1}{2}
        \|   \nabla \psi\|_{L^\infty_T(H^1(\mathbb{R}^2))}^\frac{1}{2},\label{3D-E3.1}
    \end{eqnarray}
    where   $C$ is a positive constant  independent of $T$.
\end{lem}
\begin{proof}
Using Sobolev embedding Theorem and Minkowski's inequality, we obtain
   \begin{eqnarray}
      \|\nabla \psi\|_{L^4_T(L^\infty(\mathbb{R}^2))}&\leq& C
      \left\|
      \|\nabla  \psi\|_{L^2_{x_1}}^\frac{1}{2}
            \|\nabla  \partial_1  \psi\|_{L^2_{x_1}}^\frac{1}{2}
      \right\|_{L^4_T(L^\infty_{x_2})}\nonumber\\
      &\leq&C
      \|\nabla  \psi\|_{L^\infty_T(L^\infty_{x_2}(L^2_{x_1}))}^\frac{1}{2}
            \|\nabla  \partial_1  \psi\|_{L^2_T(L^\infty_{x_2}(L^2_{x_1}))}^\frac{1}{2}\nonumber\\
      &\leq&C
      \|\nabla  \psi\|_{L^\infty_T(L^2_{x_1}(L^\infty_{x_2}))}^\frac{1}{2}
            \|\nabla  \partial_1  \psi\|_{L^2_T(L^2_{x_1}(L^\infty_{x_2}))}^\frac{1}{2}\nonumber\\
      &\leq&C
      \left\|
      \|\nabla  \psi\|_{L^2_{x_2}}^\frac{1}{2}
            \|\nabla  \partial_2\psi\|_{L^2_{x_2}}^\frac{1}{2}
      \right\|_{L^\infty_T(L^2_{x_1})}^\frac{1}{2}
        \left\|
      \|\nabla\partial_1  \psi\|_{L^2_{x_2}}^\frac{1}{2}
            \|\nabla  \partial_1\partial_2\psi\|_{L^2_{x_2}}^\frac{1}{2}
      \right\|_{L^2_T(L^2_{x_1})}^\frac{1}{2}\nonumber\\
      &\leq&C
      \|\nabla  \psi\|_{L^\infty_T(L^2(\mathbb{R}^2))}^\frac{1}{4}
            \|\nabla  \partial_2\psi\|_{L^\infty_T(L^2(\mathbb{R}^2))}^\frac{1}{4}
             \|\nabla \partial_1 \psi\|_{L^2_T(L^2(\mathbb{R}^2))}^\frac{1}{4}
            \|\nabla  \partial_1\partial_2\psi\|_{L^2_T(L^2(\mathbb{R}^2))}^\frac{1}{4}\nonumber\\
            &\leq& C      \| \partial_1\nabla \psi\|_{L^2_T(H^1(\mathbb{R}^2))}^\frac{1}{2}
        \|   \nabla \psi\|_{L^\infty_T(H^1(\mathbb{R}^2))}^\frac{1}{2}.
    \end{eqnarray}
\end{proof}

Similar to \cite{Lin12-2}, by taking divergence of the $v$ equation of (\ref{m0-E1.5-N2}), we can express the pressure function $p$ via
    \begin{equation}
      p=-2 \partial_2 \psi +\sum_{i,j=1}^2(-\Delta)^{-1}
      [\partial_iv_j\partial_j v_i+\partial_i\partial_j(\partial_i \psi\partial_j \psi)
      ].\label{m0-E3.14}
    \end{equation}
Substituting (\ref{m0-E3.14}) into (\ref{m0-E1.5-N2}), we have
    \begin{equation}
  \left\{
  \begin{array}{l}
    \partial_t \psi +v\cdot\nabla \psi + v_2=0,\ \ (t,x)\in\mathbb{R}^+\times\mathbb{R}^2,\\
        \partial_t v_1+v\cdot\nabla v_1 -\Delta v_1 - \partial_1\partial_2\psi =f_1,\\
        \partial_t v_2+v\cdot\nabla v_2 -\Delta v_2 + \partial_1^2\psi =f_2,\\
            \mathrm{div}v=0,\\
            (\psi,v)|_{t=0}=(\psi_0,v_0),
  \end{array}
  \right.\label{m0-E3.15}
\end{equation}
where
    $$
    f=(f_1,f_2)^\top=-\displaystyle{\sum_{i,j=1}^2}\nabla(-\Delta)^{-1}
      [\partial_iv_j\partial_j v_i+\partial_i\partial_j(\partial_i \psi\partial_j \psi)
      ]-\displaystyle{\sum_{j=1}^2}\partial_j[\nabla \psi\partial_j \psi].
    $$

The next lemma is a standard  energy estimate.
\begin{lem}\label{m0-L2.6}
 Let   $(\psi,v)$ be sufficiently smooth functions which solve (\ref{m0-E1.5-N2}),
  then there holds
    \begin{eqnarray}\label{m0-E2.5-00}
   &&\frac{d}{dt}\left\{
   \frac{1}{2}\left(
   \|v\|_{H^2}^2+\|\nabla \psi\|_{H^2}^2+\frac{1}{4 }\|\Delta\psi\|_{H^1}^2
   \right)+\frac{1}{4}(v_2|\Delta\psi)_{H^1}
   \right\}       \\
        &&+\|\nabla v \|_{H^2}^2-\frac{1 }{4}\|\nabla v_2\|_{H^1}^2
        +\frac{1 }{4}\|\nabla\partial   _1\psi\|_{H^1}^2\nonumber\\
   &=&-(v\cdot\nabla v|v)_{H^2}
   +(v\cdot\nabla\psi|\Delta\psi)_{H^2}
   -(\mathrm{div}(\nabla\psi\otimes\nabla\psi)|v)_{H^2}
   -\frac{1}{4}(v\cdot\nabla v_2|\Delta\psi)_{H^1}
   \nonumber\\
        &&
   +\frac{1}{4}(f_2|\Delta\psi)_{H^1}
   +\frac{1}{4}(\nabla v_2|\nabla (v\cdot\nabla \psi))_{H^1}
   -\frac{1}{4 }(\Delta(v\cdot\nabla\psi)|\Delta\psi)_{H^1} .
   \nonumber
    \end{eqnarray}
\end{lem}
\begin{proof}
Taking the standard $H^2$ inner product of (\ref{m0-E1.5-N2})$_{2,3}$ and $v$, using the integration by parts, we have
    \begin{eqnarray}
      &&\frac{1}{2}\frac{d}{dt}\|v\|_{H^2}^2
      +(v\cdot\nabla v|v)_{H^2}
            +\|\nabla v\|_{H^2}^2\nonumber\\
            &=&-( \partial_1\partial_2\psi|v_1)_{H^2}
            -( (\Delta+\partial_2^2)\psi|v_2)_{H^2}
            -(\mathrm{div}(\nabla\psi\otimes\nabla \psi)|v)_{H^2}.\label{m0-E2.11}
    \end{eqnarray}
Using the fact that $\mathrm{div}v=0$ and the integration by parts, we get
        \begin{eqnarray}
        & &-( \partial_1\partial_2\psi|v_1)_{H^2}
            = ( \partial_2 \psi|\partial_1 v_1)_{H^2}\nonumber\\
    & =& -(  \partial_2 \psi|\partial_2v_2)_{H^2}
          =  (  \partial_2^2 \psi| v_2)_{H^2}.
    \end{eqnarray}
From (\ref{m0-E1.5-N2})$_1$, we obtain
        \begin{eqnarray}
        & &-( \Delta\psi|v_2)_{H^2}
            = (  \Delta \psi|(\partial_t\psi+v\cdot\nabla \psi ))_{H^2}\nonumber\\
    & =&-\frac{1}{2}\frac{d}{dt}\|\nabla\psi\|_{H^2}^2+
    (  \Delta \psi| v\cdot\nabla \psi )_{H^2} .\label{m0-E2.13}
    \end{eqnarray}
From (\ref{m0-E2.11})-(\ref{m0-E2.13}), we have
    \begin{eqnarray}
      &&\frac{1}{2}\frac{d}{dt}\left(\|v\|_{H^2}^2
      +\|\nabla\psi\|_{H^2}^2\right)
      +\|\nabla v\|_{H^2}^2\nonumber\\
            &=& -(v\cdot\nabla v|v)_{H^2}
         +
    (  \Delta \psi| v\cdot\nabla \psi )_{H^2}
            -(\mathrm{div}(\nabla\psi\otimes\nabla \psi)|v)_{H^2}.\label{m0-E2.14}
    \end{eqnarray}
Taking the standard $H^1$ inner product of (\ref{m0-E3.15})$_{3}$ and $\Delta\psi$,
 using the integration by parts, we have
    \begin{eqnarray}
      (\partial_tv_2|\Delta\psi)_{H^1}+(v\cdot\nabla v_2|\Delta\psi)_{H^1}
      -(\Delta v_2|\Delta\psi)_{H^1}
     &=&  - \|\partial_1\nabla \psi\|_{H^1}^2
     +(f_2|\Delta\psi)_{H^1} .\label{m0-E2.15}
    \end{eqnarray}
From (\ref{m0-E1.5-N2})$_1$, using the integration by parts, we get
    \begin{eqnarray}
      &&(\partial_tv_2|\Delta\psi)_{H^1}\nonumber\\
            &=&\frac{d}{dt}(v_2|\Delta\psi)_{H^1}
            -(v_2|\Delta\partial_t\psi)_{H^1}\nonumber\\
      &=&\frac{d}{dt}(v_2|\Delta\psi)_{H^1}
           +(v_2|\Delta(v\cdot\nabla\psi + v_2))_{H^1}\nonumber\\
      &=&\frac{d}{dt}(v_2|\Delta\psi)_{H^1}
           -(\nabla v_2|\nabla(v\cdot\nabla\psi))_{H^1}
           -\|\nabla v_2\|_{H^1}^2.
    \end{eqnarray}
We observe, by (\ref{m0-E1.5-N2})$_1$,   that
    \begin{eqnarray}
      &&-(\Delta v_2|\Delta\psi)_{H^1}\nonumber\\
            &=&(\Delta(\partial_t\psi+v\cdot\nabla\psi )
            |\Delta\psi)_{H^1}\nonumber\\
      &=&\frac{1}{2 }\frac{d}{dt}\|\Delta\psi\|_{H^1}^2
      +(\Delta( v\cdot\nabla\psi )
            |\Delta\psi)_{H^1}.\label{m0-E2.17}
    \end{eqnarray}
From (\ref{m0-E2.15})-(\ref{m0-E2.17}), we have
    \begin{eqnarray}
      &&\frac{d}{dt}\left\{
      \frac{1}{2}\|\Delta\psi\|_{H^1}^2+(v_2|\Delta\psi)_{H^1}
      \right\}+ \|\partial_1\nabla\psi\|_{H^1}^2- \|\nabla v_2\|_{H^1}^2\nonumber\\
            &=&-(v\cdot\nabla v_2|\Delta\psi)_{H^1}
                        +(f_2|\Delta\psi)_{H^1}+(\nabla v_2|\nabla(v\cdot\nabla \psi))_{H^1}
            -(\Delta(v\cdot\nabla\psi)|\Delta\psi) .\label{m0-E2.18}
    \end{eqnarray}
With (\ref{m0-E2.14}) and (\ref{m0-E2.18}), one can finish the proof.
\end{proof}

The following is the key \textit{a priori} estimate which is essential to the proof
of the main result of this paper.
 \begin{lem}\label{m0-E3.3-00}
 Let   $(\psi,v)$ be sufficiently smooth functions which solve (\ref{m0-E1.5-N2}) and  satisfy
  $(  \psi, v,\nabla p)\in E^2_T$,   then there holds
    \begin{equation}
   A_{1,T}^2    \leq C
               (\|v_0\|_{H^2(\mathbb{R}^2)}^2+\|\nabla\psi_0\|_{H^2(\mathbb{R}^2)}^2)
               + C  A_T^3(1+ A_T)^2.\label{m0-E3.17}
    \end{equation}
    where   $C$ is a positive constant  independent of $T$.
\end{lem}
\begin{proof}
From  the energy estimate (\ref{m0-E2.5-00}) and the definition of $A_{1,T}$,  we obtain for a positive constant $C$ (independent of $T$) that
 \begin{eqnarray}
      A_{1,T}^2&\leq&CA_{1,0}^2+C\left|\int^T_0(v\cdot\nabla v|v)_{H^2}dt\right|
   +C\left|\int^T_0(v\cdot\nabla\psi|\Delta\psi)_{\dot{H}^1}+(v\cdot\nabla\psi|\Delta\psi)_{\dot{H}^2}dt\right|
   \nonumber\\
        &&
   +C\left|\int^T_0(\mathrm{div}(\nabla\psi\otimes\nabla\psi)|v)_{\dot{H}^1}+(\mathrm{div}(\nabla\psi\otimes\nabla\psi)|v)_{\dot{H}^2}dt\right| +C\left|\int^T_0(v\cdot\nabla v_2|\Delta\psi)_{H^1}dt\right|\nonumber\\
        &&
    +C\left|\int^T_0(f_2|\Delta\psi)_{H^1}dt\right|
   +C\left|\int^T_0(\nabla v_2|\nabla (v\cdot\nabla \psi))_{H^1}dt\right|
   +C\left|\int^T_0(\Delta(v\cdot\nabla\psi)|\Delta\psi)_{H^1} dt\right|\nonumber\\
   &=&CA_{1,0}^2+\sum_{j=1}^7H_j.\label{m0-E3.18-00}
    \end{eqnarray}

  We are going to estimate term by term the right hand side of the above
inequality. The basic strategies
are the similar to \cite{Lin14}. More precisely, we estimate separately terms involving
 $\partial_1 \psi$ and terms with  $\partial_2 \psi$. For terms with
  $\partial_1 \psi$, one can use the dissipations implied by the energy
equality (\ref{m0-E2.5-00}). For terms containing $\partial_2\psi$, we
use the algebriac relation (deduced from that $\mathrm{div}v=0$) and the
transport equations. The latter reduces space-time estimates to bounds
on time-slices and terms with either $\partial_1 \psi$ or of higher
order nonlinearities (hence they are smaller under our smallness
assumptions on the initial data).
       To illustrate the basic idea, we start with the second term $H_2$.
Applying
H\"{o}lder and Sobolev inequalities, we deduce that
    \begin{eqnarray}
     && H_2   = C\left|\sum_{k=1}^2\sum_{|\alpha|=k}\sum_{i=1}^2\int^T_0\int
            [\partial^\alpha\partial_i(v\cdot\nabla\psi)-v\cdot\nabla\partial^\alpha\partial_i\psi]\partial^\alpha\partial_i\psi dxdt\right|
            \nonumber\\
     &\leq&
      C\|\nabla\partial_1\psi\|_{L^2_T(L^2(\mathbb{R}^2))}(\|\nabla^2 v\|_{L^2_T(L^2(\mathbb{R}^2))}\|\nabla\psi\|_{L^\infty_T(L^\infty(\mathbb{R}^2))}
     +\|\nabla  v\|_{L^2_T(L^4(\mathbb{R}^2))}\|\nabla^2\psi\|_{L^\infty_T(L^4(\mathbb{R}^2))})\nonumber\\
        &&+C\|\partial_1\psi\|_{L^2_T(L^\infty(\mathbb{R}^2))}\|\nabla^2 v_1\|_{L^2_T(L^2(\mathbb{R}^2))}\|\nabla^2\psi\|_{L^\infty_T(L^2(\mathbb{R}^2))}
        +C\left|\int^T_0\int(\partial_2^2v_2\partial_2 \psi+2\partial_2 v_2\partial_2^2\psi)\partial_2^2\psi
         dxdt\right| \nonumber\\
     &&+C\|\partial_1 \nabla^2\psi\|_{L^2_T(L^2(\mathbb{R}^2))}(\|\nabla^3 v\|_{L^2_T(L^2(\mathbb{R}^2))}\|\nabla\psi\|_{L^\infty_T(L^\infty(\mathbb{R}^2))}
            +\|\nabla^2 v\|_{L^2_T(L^4(\mathbb{R}^2))}\|\nabla^2\psi\|_{L^\infty_T(L^4(\mathbb{R}^2))}\nonumber\\
                &&
          +\|\nabla  v\|_{L^2_T(L^\infty(\mathbb{R}^2))}\|\nabla^3 \psi\|_{L^\infty_T(L^2(\mathbb{R}^2))}
     )\nonumber\\
             &&+C\|\nabla^3\psi\|_{L^\infty_T(L^2(\mathbb{R}^2))}
             (\|\nabla^3 v_1\|_{L^2_T(L^2(\mathbb{R}^2))}\|\partial_1\psi\|_{L^2_T(L^\infty(\mathbb{R}^2))}
            +\|\nabla^2 v_1\|_{L^2_T(L^4(\mathbb{R}^2))}\|\nabla\partial_1\psi\|_{L^2_T(L^4(\mathbb{R}^2))}\nonumber\\
                &&
          +\|\nabla  v_1\|_{L^2_T(L^\infty(\mathbb{R}^2))}\|\nabla^2\partial_1 \psi\|_{L^2_T(L^2(\mathbb{R}^2))}
     )\nonumber\\
        &&+C\left|\int^T_0\int \partial_2^3\psi(\partial_2^3v_2\partial_2\psi+
        3\partial_2^2v_2\partial_2^2\psi
        +3\partial_2v_2\partial_2^3\psi)dxdt\right|.   \label{m0-E3.18-0001}
    \end{eqnarray}
    Then, to estimate the  terms in the above inequality, we  need the following two lemmas, and give the detail proofs in the later.
    \begin{lem}\label{m0-L3.5}
 Under the conditions in Lemma \ref{m0-E3.3-00}, then there hold
    \begin{equation}
   \left|\int^T_0\int  \partial_2\psi  \partial_2^3\psi\partial_2^3 v_2  dxdt
    \right|
    +\left|\int^T_0\int  \partial_2^2\psi  \partial_2^3\psi\partial_2^2 v_2  dxdt
    \right|+\left|\int^T_0\int  \partial_2\psi  \partial_2^2\psi\partial_2^2 v_2  dxdt
    \right|
   \leq CA_T^3,\label{m0-E3.10}
    \end{equation}
        \begin{equation}
     \left|\int^T_0\int  \partial_2 \psi  \partial_2^3\psi\partial_2^2 v_2\partial_2^2\psi  dxdt
    \right|
    \leq CA_T^4 ,\label{m0-E3.11}
    \end{equation}
     \begin{equation}
    \left|\int^T_0\int  (\partial_2^2\psi)^2\partial_2  v_2  dxdt
    \right|
    +\left|\int^T_0\int  (\partial_2^2\psi )^2 \partial_2^3  v_2  dxdt
    \right|
     \leq  CA_T^3, \label{m0-E3.12}
    \end{equation}
    where   $C$ is a positive constant  independent of $T$.
\end{lem}
    \begin{lem}\label{m0-L3.6}
 Under the conditions in Lemma \ref{m0-E3.3-00}, then there holds
    \begin{eqnarray}
   \left|\int^T_0\int  \partial_2 v_2(\partial_2^3\psi)^2 dxdt
    \right|
    &\leq&  C  A_T^3(1+A_T^2),
    \end{eqnarray}
    where   $C$ is a positive constant  independent of $T$.
\end{lem}

Accepting these two lemmas, we proceed with our proof.
From (\ref{m0-E3.18-0001}) and Lemmas  \ref{m0-L3.5}-\ref{m0-L3.6},   we get
    \begin{eqnarray}
       H_2   &\leq& C A_T^3(1 +A_T)^2.  \label{m0-E3.18}
    \end{eqnarray}
Similarly, one can get
        \begin{equation}
     H_7 = C\left|  \sum_{i=1}^2
            \int^T_0 (\partial_i(v\cdot\nabla\psi)|
            \partial_i\Delta\psi)_{H^1}dt\right|
            \leq C A_T^3(1+ A_T)^2,\label{m0-E3.19}
    \end{equation}
and
    \begin{eqnarray}
      H_3    &=&C\left|\sum_{k=1}^2\sum_{|\alpha|=k}\sum_{i,j=1}^2\int^T_0\int
            \partial^\alpha(\partial_i\psi\partial_j\psi)\partial^\alpha\partial_i v_j dxdt\right|
            \nonumber\\
      &\leq&  \|\nabla^2 v\|_{L^2_T(L^2(\mathbb{R}^2))}(\|\partial_1\psi\|_{L^2_T(L^\infty(\mathbb{R}^2))}
      \|\nabla^2\psi\|_{L^\infty_T(L^2(\mathbb{R}^2))}
      +\|\nabla\psi\|_{L^\infty_T(L^\infty(\mathbb{R}^2))}\|\nabla\partial_1\psi\|_{L^2_T(L^2(\mathbb{R}^2))})
      \nonumber\\
      &&+C\left|\int^T_0  \partial_2\psi\partial_2^2\psi \partial_2^2 v_2 dxdt\right|\nonumber\\
                &&+C\|\nabla^3v\|_{L^2_T(L^2(\mathbb{R}^2))}
                (\|\nabla^2\psi\|_{L^\infty_T(L^4(\mathbb{R}^2))} \|\nabla\partial_1\psi\|_{L^2_T(L^4(\mathbb{R}^2))}
                +\|\partial_1\psi\|_{L^2_T(L^\infty(\mathbb{R}^2))}\|\nabla^3\psi\|_{L^\infty_T(L^2(\mathbb{R}^2))}
                \nonumber\\
      &&+\|\nabla\psi\|_{L^\infty_T(L^\infty(\mathbb{R}^2))}\|\nabla^2\partial_1\psi\|_{L^2_T(L^2(\mathbb{R}^2))}
                )+C\left|\int^T_0\int\partial_2^3v_2((\partial_2^2\psi)^2+\partial_2\psi\partial_2^3\psi) dxdt\right|
                \nonumber\\
    &\leq& CA_T^3.\label{m0-E3.20}
    \end{eqnarray}
  Apply the same line of arguments, one could deduce that
    \begin{eqnarray}
     &&\left|\sum_{i,j=1}^2\int^T_0
      (\partial_2(-\Delta)^{-1}(\partial_iv_j\partial_j v_i)|\Delta\psi)_{H^1}dt\right|\nonumber\\
            &\leq&C\int^T_0\|(\nabla v)^2\|_{H^1(\mathbb{R}^2)}\|\nabla\psi\|_{H^1(\mathbb{R}^2)} dt\nonumber\\
      &\leq&  C\|\nabla v\|_{L^2_T(H^2(\mathbb{R}^2))}^2\|\nabla\psi\|_{L^\infty_T(H^2(\mathbb{R}^2))}, \label{m0-E3.21}
    \end{eqnarray}
        \begin{eqnarray}
    &&\left|\int^T_0
    \left(\left.-\sum_{i,j=1}^2\partial_2(-\Delta)^{-1}(
    \partial_i\partial_j(\partial_i\psi\partial_j\psi))-
    \sum_{j=1}^2\partial_j(\partial_2\psi\partial_j\psi)\right|\Delta\psi
    \right)_{H^1}
    dt\right|\nonumber\\
        &=&\left|\int^T_0
    \left(\left.- \partial_2(-\Delta)^{-1}(
    \partial_1^2(\partial_1\psi\partial_1\psi))  \right|\Delta\psi
    \right)_{H^1}
    dt\right|\nonumber\\
        &&+\left|\int^T_0
    \left(\left.   2\partial_2(-\Delta)^{-1} (
    \partial_2(\partial_2\psi\partial_1\psi))+
      (\partial_2\psi\partial_1\psi)\right|\partial_1\Delta\psi
    \right)_{H^1}
    dt\right|\nonumber\\
    &&+\left|\int^T_0
     ( -  (-\Delta)^{-1}
    \partial_2^3 (\partial_2\psi)^2  -
     \partial_2(\partial_2\psi )^2 |\Delta\psi
     )_{H^1}dt\right|\nonumber\\
            &\leq& C\int^T_0
            \|\nabla (\partial_1\psi)^2\|_{H^1(\mathbb{R}^2)}\|\Delta\psi\|_{H^1(\mathbb{R}^2)}dt
            +C\int^T_0
            \|\nabla (\nabla \psi\partial_1\psi) \|_{H^1(\mathbb{R}^2)}\|\nabla\partial_1\psi\|_{H^1(\mathbb{R}^2)}dt
            \nonumber\\
    &&+\left|\int^T_0
     (    (-\Delta)^{-1}
    \partial_2\partial_1^2 (\partial_2\psi)^2 |\Delta\psi
     )_{H^1}dt\right|\nonumber\\
            &\leq& C\|\nabla\partial_1 \psi\|_{L^2_T(H^1 (\mathbb{R}^2))}(\|\nabla\partial_1 \psi\|_{L^2_T(H^1 (\mathbb{R}^2))}+\| \partial_1 \psi\|_{L^2_T(L^\infty (\mathbb{R}^2))})\|\nabla\psi\|_{L^\infty_T(H^2(\mathbb{R}^2))}\nonumber\\
            &&
            +C\int^T_0\|\partial_1(\partial_2\psi)^2\|_{H^1(\mathbb{R}^2)}\|\nabla\partial_1\psi\|_{H^1(\mathbb{R}^2)}dt\nonumber\\
            &\leq&  C\|\nabla\partial_1 \psi\|_{L^2_T(H^1 (\mathbb{R}^2))}(\|\nabla\partial_1 \psi\|_{L^2_T(H^1 (\mathbb{R}^2))}+\| \partial_1 \psi\|_{L^2_T(L^\infty (\mathbb{R}^2))})\|\nabla\psi\|_{L^\infty_T(H^2(\mathbb{R}^2))}.
            \label{m0-E3.22}
        \end{eqnarray}
With (\ref{m0-E3.21})-(\ref{m0-E3.22}), we conclude
    \begin{equation}
      H_5=C\left| \int^T_0
       (f_2|\Delta\psi)_{H^1}dt\right|\leq CA_T^3.
             \label{m0-E3.23}
    \end{equation}
In the same way, we can easily obtain the following estimates and omit the details,
    \begin{equation}
  H_1=C\left|  \int^T_0(v\cdot\nabla v|v)_{H^2}dt\right|\leq C\|\nabla v\|_{L^2_T(H^2(\mathbb{R}^2))}^2\|v\|_{L^\infty_T(H^2(\mathbb{R}^2))},
    \end{equation}
        \begin{eqnarray}
       H_4&=&C\left| \int^T_0 (v\cdot\nabla v_2|\Delta\psi)_{H^1}dt\right|\\
       &=&C\left|\sum_{j,k=1}^2 \int^T_0 (\partial_k(v_j v_2)|\partial_k\partial_j\psi)_{H^1}dt\right|\nonumber\\
           &\leq& C\|\nabla v\|_{L^2_T(H^1 (\mathbb{R}^2))}\|v\|_{L^2_T(L^\infty(\mathbb{R}^2))}
           \|\nabla^2 \psi\|_{L^\infty_T(H^1(\mathbb{R}^2))}
           +C\|\nabla v\|_{L^2_T(L^4(\mathbb{R}^2))}^2\|\nabla^3 \psi\|_{L^\infty_T(L^2(\mathbb{R}^2))}\nonumber
        \end{eqnarray}
        \begin{eqnarray}
         H_6&=&C\left| \int^T_0 (\nabla v_2|\nabla (v\cdot\nabla \psi))_{H^1}dt\right|\label{m0-E3.26}\\
         &\leq& C\|\nabla v\|_{L^2_T(H^2(\mathbb{R}^2))}^2\|\nabla \psi\|_{L^\infty_T(H^2(\mathbb{R}^2))}
         +C\|\nabla v_2\|_{L^2_T(H^1(\mathbb{R}^2))}
         \|v\|_{L^2_T(L^\infty(\mathbb{R}^2))}\|\nabla^2 \psi\|_{L^\infty_T(H^1(\mathbb{R}^2))}.\nonumber
            \end{eqnarray}
From
(\ref{m0-E3.18-00}), (\ref{m0-E3.18})-(\ref{m0-E3.20}) and (\ref{m0-E3.23})-(\ref{m0-E3.26}),   one can obtain (\ref{m0-E3.17}).
\end{proof}

At the end of this section, we give the proofs of two technical Lemmas  \ref{m0-L3.5}-\ref{m0-L3.6} that are needed in establishing the key \textit{a priori} estimates. Using the condition $\mathrm{div}v=0$, we can replace $\partial_1v_1$ by $\partial_2 v_2$ in various calculations in the proof of Lemma \ref{m0-L3.5}.

\noindent\textbf{Proof of Lemma \ref{m0-L3.5}.}
  Using the fact that $\mathrm{div} v=0$, the integration by parts,  H\"{o}lder's inequality and Sobolev embedding Theorem,
  we have
    \begin{eqnarray*}
      &&\left|\int^T_0\int  \partial_2\psi  \partial_2^3\psi\partial_2^3 v_2  dxdt
    \right|\\
            &=&\left|\int^T_0\int  \partial_2\psi  \partial_2^3\psi\partial_2^2\partial_1 v_1  dxdt
    \right|\\
    &=&\left| \int^T_0\int \left(-\partial_2^2  v_1 \partial_1\partial_2\psi  \partial_2^3\psi
    -\partial_2\psi \partial_2^2  v_1 \partial_1\partial_2^3\psi
    \right)  dxdt
    \right|\\
    &=&\left| \int^T_0\int \left(-\partial_2^2  v_1 \partial_1\partial_2\psi  \partial_2^3\psi
    +\partial_2^2\psi \partial_2^2  v_1 \partial_1\partial_2^2\psi
    +\partial_2\psi \partial_2^3  v_1\partial_1\partial_2^2\psi
    \right)  dxdt
    \right|\\
    &\leq& C\|\partial_2^2  v_1\|_{L^2_T(L^4(\mathbb{R}^2))}\|\partial_1\partial_2 \psi\|_{L^2_T(L^4(\mathbb{R}^2))}
    \|\partial_2^3\psi\|_{L^\infty_T(L^2(\mathbb{R}^2))}\\
    &&+C\|\partial_2^2\psi\|_{L^\infty_T(L^4(\mathbb{R}^2))}
    \|\partial_2^2 v_1\|_{L^2_T(L^4(\mathbb{R}^2))}\|\partial_1\partial_2^2\psi\|_{L^2_T(L^2(\mathbb{R}^2))}\\
            &&
    +C\|\partial_2\psi\|_{L^\infty_T(L^\infty(\mathbb{R}^2))}
    \|\partial_2^3 v_1\|_{L^2_T(L^2(\mathbb{R}^2))}\|\partial_1\partial_2^2\psi\|_{L^2_T(L^2(\mathbb{R}^2))}\\
            &\leq& C\|\nabla\psi\|_{L^\infty_T(H^2(\mathbb{R}^2))}\|\partial_1\nabla\psi\|_{L^2_T(H^1(\mathbb{R}^2))}
            \|\nabla v\|_{L^2_T(H^2(\mathbb{R}^2))}.
    \end{eqnarray*}
Similarly, we can obtain (\ref{m0-E3.10})-(\ref{m0-E3.11}).   Using the fact that $\mathrm{div} v=0$, the integration by parts,  H\"{o}lder's inequality and Sobolev embedding Theorem,
  we obtain
    \begin{eqnarray*}
      &&\left|\int^T_0\int  (\partial_2^2\psi)^2  \partial_2  v_2  dxdt
    \right|\\
            &=&\left|\int^T_0\int  (\partial_2^2\psi)^2  \partial_1 v_1  dxdt
    \right|\\
    &=&\left| \int^T_0\int  2   v_1 \partial_1\partial_2^2\psi  \partial_2^2\psi  dxdt
    \right|\\
    &\leq& C\|  v_1\|_{L^2_T(L^\infty(\mathbb{R}^2))}\|\partial_1\partial_2^2 \psi\|_{L^2_T(L^2(\mathbb{R}^2))}
    \|\partial_2^2\psi\|_{L^\infty_T(L^2(\mathbb{R}^2))}.
    \end{eqnarray*}
    Similarly, we can obtain (\ref{m0-E3.12}).
 {\hfill
$\square$\medskip}

From the idea in Remark \ref{m0-Rem1.2}, we can give the proof of the key lemma \ref{m0-L3.6}.
Here, it would
be useful (and it may be also necessary) to replace $v_2$. In fact,
via (\ref{m0-E1.5-N2})$_1$, we can re-write
    \begin{equation}
      v_2=-(\partial_t\psi+v\cdot\nabla \psi ). \label{m0-E3.9}
    \end{equation}
The above substitution for $v_2$  has the advantage that it reduces
space-time integral estimates to estimates on time slices and space times
integral with higher order nonlinearities and fast dissipation. The latter
is smaller by the initial smallness assumptions.
To prove Lemma \ref{m0-L3.6}, we give the following lemma, where we use the equation (\ref{m0-E3.9})   to bounded the term $\int^T_0\int \partial_2\psi \partial_2 v_2(\partial_2^3\psi)^2 dxdt$.
\begin{lem}\label{m0-L3.4}
 Under the conditions of Lemma \ref{m0-E3.3-00}, then there holds
    \begin{eqnarray}
     \left|\int^T_0\int  \partial_2\psi \partial_2 v_2(\partial_2^3\psi)^2 dxdt
    \right|
    &\leq& CA_T^4(1+A_T), \label{m0-E3.8}
    \end{eqnarray}
    where   $C$ is a positive constant  independent of $T$.
\end{lem}
\begin{proof}
From Lemma \ref{m0-L3.2}, (\ref{m0-E3.9}), the integration by parts,  H\"{o}lder's inequality and Sobolev embedding Theorem, we get
    \begin{eqnarray*}
      &&\left|\int^T_0\int  \partial_2\psi \partial_2 v_2(\partial_2^3\psi)^2 dxdt
    \right|\\
    &=&\left|\int^T_0\int   \partial_2\psi \partial_2
    (\partial_t\psi+v\cdot\nabla \psi )(\partial_2^3\psi)^2 dxdt
    \right|\\
    &=&\left| \left.\int \frac{1}{2 } (\partial_2\psi)^2 (\partial_2^3\psi)^2 dx\right|^T_0
    +\int^T_0\int \left[ - (\partial_2\psi)^2 \partial_2^3\psi\partial_2^3\partial_t\psi
    + \partial_2\psi \partial_2
    (v\cdot\nabla \psi )(\partial_2^3\psi)^2 \right]dxdt
     \right|\\
        &\leq& C\|\partial_2\psi\|_{L^\infty_T(L^\infty(\mathbb{R}^2))}^2\|\partial_2^3\psi\|_{L^\infty_T(L^2(\mathbb{R}^2))}^2 \\
        && +\left|
    \int^T_0\int \left[  (\partial_2\psi)^2 \partial_2^3\psi\partial_2^3(v\cdot\nabla\psi + v_2)
    +\partial_2\psi\partial_2(v\cdot\nabla\psi)   (\partial_2^3\psi)^2 \right]dxdt
     \right|\\
      &\leq& C\|\nabla\psi\|^4_{L^\infty_T(H^2(\mathbb{R}^2))}
       +C\|\partial_2\psi\|_{L^4_T(L^\infty(\mathbb{R}^2))}^2\|\partial_2^3\psi\|_{L^\infty_T(L^2(\mathbb{R}^2))}
         \|\partial_2^3(v\cdot\nabla\psi)-v\cdot\nabla\partial_2^3\psi\|_{L^2_T(L^2(\mathbb{R}^2))} \\
        &&+C\|\partial_2\psi\|_{L^4_T(L^\infty(\mathbb{R}^2))}^2\|\partial_2^3\psi\|_{L^\infty_T(L^2(\mathbb{R}^2))}
        \|\partial_2^3v_2        \|_{L^2_T(L^2(\mathbb{R}^2))} \\
        &&+\left|
    \int^T_0\int \left\{\frac{1}{2 }
     \left[(\partial_2\psi)^2  v\cdot\nabla(\partial_2^3\psi)^2
   +
     v\cdot\nabla (\partial_2\psi)^2 (\partial_2^3\psi)^2\right] \right\}dxdt
     \right|\\
     &&+\left|
    \int^T_0\int \left\{ \partial_2\psi\partial_2 v\cdot\nabla \psi  (\partial_2^3\psi)^2 \right\}dxdt
     \right|\\
              &\leq& CA_T^4 (1+A_T) ,
    \end{eqnarray*}
where we use the estimates
    \begin{eqnarray*}
      &&\|\partial_2^3(v\cdot\nabla\psi)-v\cdot\nabla\partial_2^3\psi\|_{L^2_T(L^2(\mathbb{R}^2))}\\
        &\leq&C  \|\nabla^3 v\|_{L^2_T(L^2(\mathbb{R}^2))}\|\nabla \psi\|_{L^\infty_T(L^\infty(\mathbb{R}^2))}
        +C\|\nabla^2 v\|_{L^2_T(L^4(\mathbb{R}^2))}\|\nabla^2 \psi\|_{L^\infty_T(L^4(\mathbb{R}^2))}
        \\
        &&+C\|\nabla v\|_{L^2_T(L^\infty(\mathbb{R}^2))}\|\nabla^3 \psi\|_{L^\infty_T(L^2(\mathbb{R}^2))}\\
        &\leq&C\|\nabla v\|_{L^2_T(H^2(\mathbb{R}^2))}\|\nabla  \psi\|_{L^\infty_T(H^2(\mathbb{R}^2))},
    \end{eqnarray*}
and
    \begin{eqnarray*}
      &&\left|
    \int^T_0\int \left\{ \partial_2\psi\partial_2 v\cdot\nabla \psi  (\partial_2^3\psi)^2 \right\}dxdt
     \right|
     \\
     &\leq& \|\nabla\psi\|_{L^4_T(L^\infty(\mathbb{R}^2))}^2\|\nabla v\|_{L^2_T(L^\infty(\mathbb{R}^2))}\|\nabla^3\psi\|_{L^\infty_T(L^2(\mathbb{R}^2))}^2.
    \end{eqnarray*}
\end{proof}

\noindent\textbf{Proof of Lemma \ref{m0-L3.6}.}
    From Lemmas \ref{m0-L3.2}, \ref{m0-L3.5}, \ref{m0-L3.4}, (\ref{m0-E3.9}),  the integration by parts,  H\"{o}lder's inequality and Sobolev embedding Theorem, we get
    \begin{eqnarray*}
      &&\left|\int^T_0\int  \partial_2 v_2(\partial_2^3\psi)^2 dxdt
    \right|\\
        &=&\left|\int^T_0\int  \partial_2 (\partial_t
        \psi+v\cdot\nabla\psi )(\partial_2^3\psi)^2 dxdt
    \right|\\
        &=&\left|\left.\int  \partial_2
        \psi (\partial_2^3\psi)^2 dx\right|^T_0
        +\int^T_0\int \left\{  -2\partial_2\psi\partial_2^3\psi\partial_2^3\partial_t\psi
         + \partial_2 ( v\cdot\nabla\psi )(\partial_2^3\psi)^2\right\}dxdt
    \right|\\
        &\leq& C\|\partial_2\psi\|_{L^\infty_T(L^\infty(\mathbb{R}^2))}\|\partial_2^3\psi\|_{L^\infty_T(L^2(\mathbb{R}^2))}^2 \\
            &&  +\left|
        \int^T_0\int \left\{  2\partial_2\psi\partial_2^3\psi\partial_2^3(v\cdot\nabla\psi+
        v_2) + \partial_2 ( v\cdot\nabla\psi )(\partial_2^3\psi)^2\right\}dxdt
    \right|\\
        &\leq& C\|\nabla\psi\|_{L^\infty_T(H^2(\mathbb{R}^2))}^3
        +C\|\nabla \psi\|_{L^\infty_T(H^2(\mathbb{R}^2))}\|\nabla\partial_1\psi\|_{L^2_T(H^1(\mathbb{R}^2))}\|\nabla v\|_{L^2_T(H^2(\mathbb{R}^2))}\\
           &&  +\left|
        \int^T_0\int   \left\{ 2\partial_2\psi\partial_2^3\psi[\partial_2^3(v\cdot\nabla\psi)-v\cdot\nabla\partial_2^3\psi]
         +
           \partial_2\psi  v \cdot\nabla(\partial_2^3\psi)^2    \right.\right.\\
        &&\left.\left. + v\cdot\nabla \partial_2\psi(\partial_2^3\psi)^2 + \partial_2   v\cdot\nabla\psi (\partial_2^3\psi)^2\right\}dxdt
    \right|\\
    &\leq&
    C A_T^3(1+A_T^2),
    \end{eqnarray*}
    where we use the estimate
    \begin{eqnarray*}
    &&   \left|
        \int^T_0\int   \left\{ 2\partial_2\psi\partial_2^3\psi[\partial_2^3(v\cdot\nabla\psi)-v\cdot\nabla\partial_2^3\psi]
          + \partial_2   v\cdot\nabla\psi (\partial_2^3\psi)^2\right\}dxdt
    \right|\\
        &\leq& \left|
        \int^T_0\int    2\partial_2\psi\partial_2^3\psi(\partial_2^3v\cdot\nabla\psi+3\partial_2^2v_1 \partial_1\partial_2\psi
        +3\partial_2^2v_2\partial_2^2\psi
         )  dxdt
    \right|\\
    &&+\left|
        \int^T_0\int   6\partial_2\psi\partial_2^3\psi
        \partial_2 v_1\partial_1\partial_2^2\psi   dxdt
    \right|
         +\left|
        \int^T_0\int  7\partial_2\psi
        \partial_2 v_2 (\partial_2^3\psi)^2   dxdt
    \right|\\
       && +\left|
        \int^T_0\int
        \partial_2 v_1\partial_1 \psi (\partial_2^3\psi)^2   dxdt
    \right|\\
        &\leq&C\|\nabla\psi\|_{L^4_T(L^{ \infty}(\mathbb{R}^2))}^2\|\partial_2^3\psi\|_{L^\infty_T(L^2(\mathbb{R}^2))}\|\nabla^3 v\|_{L^2_T(L^2(\mathbb{R}^2))}
        \\
        &&+C\|\nabla\psi\|_{L^\infty_T(L^{ \infty}(\mathbb{R}^2))}\|\partial_3^3\psi\|_{L^\infty_T(L^2(\mathbb{R}^2))}\|\nabla^2 v\|_{L^2_T(L^4(\mathbb{R}^2))}
        \|\partial_1\nabla\psi\|_{L^2_T(L^4(\mathbb{R}^2))}
        \\
            &&+C\|\nabla\psi\|_{L^\infty_T(H^2(\mathbb{R}^2))}^2\|\partial_1\nabla\psi\|_{L^2_T(H^2(\mathbb{R}^2))}
            \|\nabla v\|_{L^2_T(H^2(\mathbb{R}^2))} \\
            &&+C\|\partial_2\psi\|_{L^\infty_T(L^\infty(\mathbb{R}^2))}
            \|\partial_2^3\psi\|_{L^\infty_T(L^2(\mathbb{R}^2))}\|\partial_2 v\|_{L^2_T(L^\infty(\mathbb{R}^2))}
            \|\partial_1\partial_2^2\psi\|_{L^2_T(L^2(\mathbb{R}^2))}\\
     &&+    CA_T^4(1+A_T)  +C\|\partial_2 v\|_{L^2_T(L^\infty(\mathbb{R}^2))}
     \|\partial_1\psi\|_{L^2_T(L^\infty(\mathbb{R}^2))}\|\partial_3^3\psi\|_{L^\infty_T(L^2(\mathbb{R}^2))}^2  \\
      &\leq&  C A_T^3(1+A_T^2).
    \end{eqnarray*}
 {\hfill
$\square$\medskip}

\section{A priori estimates for $A_{2,T}$}\label{m0-S4}
   Due to Remark \ref{m0-Rem1.3}, we need to
 estimate $A_{2,T}$ as follows. At first, we estimate the easy one
 $$(\partial_1\psi,v)1_{\{|\xi|\geq1\}\cup\{|\xi|< 1,|\xi_1|\leq|\xi|^2
 \}}.$$

    \begin{lem}\label{L3.1}
 Under the conditions in Lemma \ref{m0-E3.3-00}, then there holds
    \begin{equation}
   \left\|\widehat{v}1_{\{|\xi|\geq 1\}}
    \right\|_{L^2_T(L^1_\xi)}
    +   \left\|\widehat{\partial_1\psi}1_{\{|\xi|\geq 1\}}
    \right\|_{L^2_T(L^1_\xi)}     \leq   C  A_{1,T},\label{3.1}
    \end{equation}
        \begin{equation}
   \left\|\widehat{v}1_{\{|\xi|< 1,|\xi_1|\leq|\xi|^2\}}
    \right\|_{L^2_T(L^1_\xi)}
    +   \left\|\widehat{\partial_1\psi}1_{\{|\xi|< 1,|\xi_1|\leq|\xi|^2\}}
    \right\|_{L^2_T(L^1_\xi)}     \leq   C  A_{1,T},\label{3.2}
    \end{equation}
    where   $C$ is a positive constant  independent of $T$.
\end{lem}
\begin{proof}
  Using the  H\"{o}lder's inequality, we have
     \begin{eqnarray*}
   \left\|\widehat{v}1_{\{|\xi|\geq 1\}}
    \right\|_{L^2_T(L^1_\xi)}
    +   \left\|\widehat{\partial_1\psi}1_{\{|\xi|\geq 1\}}
    \right\|_{L^2_T(L^1_\xi)}    & \leq &  C   \left\|(|\xi|^2\widehat{v} ,|\xi|^2|\xi_1|\widehat{\psi})
    \right\|_{L^2_T(L^2_\xi)} \left(\int_{|\xi|\geq1}|\xi|^{-4}d\xi\right)^\frac{1}{2}\\
    &\leq& C\|(\nabla^2 v,\nabla^2\partial_1\psi)\|_{L^2_T(L^2)},
    \end{eqnarray*}
    which finish the proof of (\ref{3.1}).
     Using the  H\"{o}lder's inequality, we have
     \begin{eqnarray*}
   &&\left\|(\widehat{v}1_{\{|\xi|< 1,|\xi_1|\leq|\xi|^2\}},\widehat{\partial_1\psi}1_{\{|\xi|< 1,|\xi_1|\leq|\xi|^2\}})
    \right\|_{L^2_T(L^1_\xi)}  \\
    &  \leq &  C   \left\|(|\xi_1|^\frac{3}{8}|\xi_2|^\frac{1}{4}\widehat{v}1_{\{|\xi|< 1,|\xi_1|\leq|\xi|^2\}},
    |\xi_1|^\frac{3}{8}|\xi_2|^\frac{1}{4}\widehat{\partial_1\psi}1_{\{|\xi|< 1,|\xi_1|\leq|\xi|^2\}})
    \right\|_{L^2_T(L^2_\xi)} \left(\int_{|\xi|<1}|\xi_1|^{-\frac{3}{4}}|\xi_2|^{-\frac{1}{2}}d\xi\right)^\frac{1}{2}\\
    &\leq& C   \left\|(|\xi|^\frac{3}{4}|\xi_2|^\frac{1}{4}\widehat{v} ,
    |\xi|^\frac{3}{4}|\xi_2|^\frac{1}{4}\widehat{\partial_1\psi} )
    \right\|_{L^2_T(L^2_\xi)}^2\leq C\|(\nabla v,\nabla\partial_1\psi)\|_{L^2_T(L^2)}
    \end{eqnarray*}
    which finish the proof of (\ref{3.2}).
\end{proof}

To estimate the difficult part $(\partial_1\psi,v)1_{ \{|\xi|< 1,|\xi_1|>|\xi|^2
 \}}$, we need obtain the following expressions of solutions.
From (\ref{m0-E3.15}), we have
    \begin{equation}
  \left\{
  \begin{array}{l}
    \partial_t \psi  + v_2=F,\ \ (t,x)\in\mathbb{R}^+\times\mathbb{R}^2,\\
        \partial_t v_1  -\Delta v_1 - \partial_1\partial_2\psi =G_1,\\
        \partial_t v_2  -\Delta v_2 + \partial_1^2\psi =G_2,\\
            \mathrm{div}v=0,\\
            (\psi,v)|_{t=0}=(\psi_0,v_0),
  \end{array}
  \right. \label{2}
\end{equation}
where
    $$
    F=-v\cdot\nabla \psi,
    $$
    $$
    G=(G_1,G_2)^\top=-v\cdot\nabla v-\displaystyle{\sum_{i,j=1}^2}\nabla(-\Delta)^{-1}
      [\partial_iv_j\partial_j v_i+\partial_i\partial_j(\partial_i \psi\partial_j \psi)
      ]-\displaystyle{\sum_{j=1}^2}\partial_j[\nabla \psi\partial_j \psi].
    $$
Using the classical Fourier analysis method, we obtain that
    \begin{equation}
      \widehat{\psi}(t,\xi)=M_{11}(t,\xi)\widehat{\psi}_0(\xi)+M_{12}(t,\xi)\widehat{v}_{0,2}(\xi)
      +\int^t_0M_{11}(t-s,\xi)\widehat{F}(\xi,s)ds+\int^t_0M_{12}(t-s,\xi)\widehat{G_2}(\xi,s)ds,\label{3.4}
    \end{equation}
        \begin{equation}
      \widehat{v}_1(t,\xi)=M_{21}(t,\xi)\widehat{\psi}_0(\xi)+M_{22}(t,\xi)\widehat{v}_{0,1}(\xi)
      +\int^t_0M_{21}(t-s,\xi)\widehat{F}(\xi,s)ds+\int^t_0M_{22}(t-s,\xi)\widehat{G_1}(\xi,s)ds,
    \end{equation}
            \begin{equation}
      \widehat{v}_2(t,\xi)=M_{31}(t,\xi)\widehat{\psi}_0(\xi)+M_{32}(t,\xi)\widehat{v}_{0,2}(\xi)
      +\int^t_0M_{31}(t-s,\xi)\widehat{F}(\xi,s)ds+\int^t_0M_{32}(t-s,\xi)\widehat{G_2}(\xi,s)ds,
    \end{equation}
where
    \begin{equation}
      M_{11}(t,\xi)=\frac{\lambda_+e^{\lambda_-t}-\lambda_-e^{\lambda_+t}}{\lambda_+-\lambda_-},
      \   M_{12}(t,\xi)=\frac{ e^{\lambda_-t}- e^{\lambda_+t}}{\lambda_+-\lambda_-},\label{3.7}
    \end{equation}
        \begin{equation}
      M_{21}(t,\xi)=\xi_1\xi_2\frac{ e^{\lambda_-t}- e^{\lambda_+t}}{\lambda_+-\lambda_-},
      \   M_{22}(t,\xi)=\frac{\lambda_+e^{\lambda_+t}-\lambda_-e^{\lambda_-t}}{\lambda_+-\lambda_-},
    \end{equation}
            \begin{equation}
      M_{31}(t,\xi)=-|\xi_1|^2\frac{ e^{\lambda_-t}- e^{\lambda_+t}}{\lambda_+-\lambda_-},
      \   M_{32}(t,\xi)=M_{22}(t,\xi),
    \end{equation}
    \begin{equation}
      \lambda_\pm=-\frac{|\xi|^2}{2}\pm i\frac{\sqrt{4\xi^2_1-\xi^4}}{2},
      \ \textrm{ when }|\xi_1|> |\xi|^2. \label{3.10}
    \end{equation}
\begin{lem}\label{L3.2}
 Under the conditions in Lemma \ref{m0-E3.3-00}, then there holds
        \begin{equation}
   \left\|\widehat{v}1_{\{|\xi|< 1,|\xi_1|>|\xi|^2\}}
    \right\|_{L^2_T(L^1_\xi)}
    +   \left\|\widehat{\partial_1\psi}1_{\{|\xi|< 1,|\xi_1|>|\xi|^2\}}
    \right\|_{L^2_T(L^1_\xi)}  \leq CA_{2,0}+ C   A_{T}^2,\label{3.3}
    \end{equation}
    where   $C$ is a positive constant  independent of $T$.
\end{lem}
\begin{proof}
 When $\xi\in A=\{|\xi_1|\geq |\xi|^2\}$,  from  (\ref{3.7}), (\ref{3.10}), we get
    \begin{equation}
      |M_{11}|=\left|e^{-\frac{|\xi|^2t}{2}}
      \frac{\xi^2\sin(\frac{\sqrt{4\xi_1^2-\xi^4}}{2}t)+\sqrt{4\xi_1^2-\xi^4}\cos
      (\frac{\sqrt{4\xi_1^2-\xi^4}}{2}t)}{\sqrt{4\xi_1^2-\xi^4}}
      \right|\leq C e^{-\frac{|\xi|^2t}{2}},\label{3.12}
    \end{equation}
     \begin{equation}
      |M_{12}|=\left|e^{-\frac{|\xi|^2t}{2}}
      \frac{ 2\sin(\frac{\sqrt{4\xi_1^2-\xi^4}}{2}t) }{\sqrt{4\xi_1^2-\xi^4}}
      \right|\leq C|\xi_1|^{-1} e^{-\frac{|\xi|^2t}{2}},\label{3.13}
    \end{equation}
        then we can estimate the linear part as follows,
            \begin{equation}
      \| M_{11}\widehat{\partial_1\psi}_0\|_{L^2_T(L^1_\xi(A))} \leq C\|\xi_1e^{-\frac{|\xi|^2t}{2}}\widehat{\psi}_0\|_{L^2_T(L^1_\xi)}\leq CA_{2,0},\label{3.14}
    \end{equation}
         \begin{equation}
      \|\xi_1M_{12}\widehat{v}_{0,2}\|_{L^2_T(L^1_\xi(A))}  \leq C\| e^{-\frac{|\xi|^2t}{2}}\widehat{v}_{0,2}\|_{L^2_T(L^1_\xi)}\leq CA_{2,0}.\label{3.15}
    \end{equation}
From (\ref{3.12})-(\ref{3.13}), we obtain
                \begin{equation}
      \|\xi_1M_{11}(t)1_{A}\|_{L^2_{\xi_2}(L^\infty_{\xi_1})} \leq C  \|\xi_1e^{-\frac{|\xi|^2t}{2}}\|_{L^2_{\xi_2}(L^\infty_{\xi_1})}
       \leq C t^{-\frac{1}{2}} \| e^{-\frac{\xi_2^2t}{2}}\|_{L^2_{\xi_2}}\leq C t^{-\frac{3}{4}},\label{3.16}
    \end{equation}
     \begin{equation}
     \|\xi_1M_{12}(t)1_{A}\|_{L^1_{\xi_2}(L^2_{\xi_1})} \leq C  \| e^{-\frac{|\xi|^2t}{2}}\|_{L^1_{\xi_2}(L^2_{\xi_1})}
       \leq C   \| e^{-\frac{\xi_1^2t}{2}}\|_{L^2_{\xi_1}} \| e^{-\frac{\xi_2^2t}{2}}\|_{L^1_{\xi_2}}\leq C t^{-\frac{3}{4}}.\label{3.17}
   \end{equation}
Using   Young's inequality and the interpolation inequality, we have
    \begin{eqnarray}
       \|\widehat{F}\|_{L^\frac{4}{3}_T(L^2_{\xi_2}(L^1_{\xi_1}))}
      &\leq &
      C\|\widehat{v} \|_{L^2_T(L^1_{\xi})}\|\widehat{\nabla\psi}\|_{L^4_T(L^2_{\xi_2}(L^1_{\xi_1}))}\nonumber\\
      &\leq&
       C\|\widehat{v} \|_{L^2_T(L^1_{\xi})}\|\widehat{\nabla\psi}\|_{L^\infty_T(L^2_{\xi})}^\frac{1}{2}\||\xi_1|\widehat{\nabla\psi}
      \|_{L^2_T(L^2_{\xi})}^\frac{1}{2}\nonumber
    \\
    &\leq&CA_{1,T}A_{2,T},
    \end{eqnarray}
and combining the idea in (\ref{m0-E3.22}),
    \begin{eqnarray}
   \|\widehat{G}_2\|_{L^\frac{4}{3}_T(L^\infty_{\xi_2}(L^2_{\xi_1}))}
      &\leq &
      C\|\widehat{v}\|_{L^4_T(L^2_{\xi_2}(L^1_{\xi_1}))}\|\widehat{\nabla v}\|_{L^2_T(L^2_{\xi})}
      +
      C\|\widehat{\partial_1\nabla \psi}\|_{L^2_T(L^2_{\xi})}\|\widehat{\nabla\psi}\|_{L^4_T(L^2_{\xi_2}(L^1_{\xi_1}))}\nonumber\\
    &\leq&CA_{1,T}^2.\label{3.19}
    \end{eqnarray}
    From (\ref{3.16})-(\ref{3.19}), using the Riesz potential inequality (\cite{Stein}, Theorem 1, P119), we have
        \begin{eqnarray}
          &&\left\|1_A\xi_1\int^t_0M_{11}(t-s,\xi)\widehat{F}(\xi,s)ds\right\|_{L^2_T(L^1_\xi)}\nonumber\\
          &\leq&  \left\|\int^t_0 \|1_A\xi_1M_{11}(t-s,\xi)\|_{L^2_{\xi_2}(L^\infty_{\xi_1})}
          \|\widehat{F}(\xi,s)\|_{L^2_{\xi_2}(L^1_{\xi_1})}ds\right\|_{L^2_T}\nonumber\\
           &\leq& C \left\|\int^t_0 (t-s)^{-\frac{3}{4}}
          \|\widehat{F}(\xi,s)\|_{L^2_{\xi_2}(L^1_{\xi_1})}ds\right\|_{L^2_T}\nonumber\\
          &\leq&C\left\|\widehat{F}(\xi,s)\right\|_{L^\frac{4}{3}_T(L^2_{\xi_2}(L^1_{\xi_1}))}\leq CA_{1,T}A_{2,T},\label{3.20}
        \end{eqnarray}
    and
            \begin{eqnarray}
          &&\left\|1_A\xi_1\int^t_0M_{12}(t-s,\xi)\widehat{G}_2(\xi,s)ds\right\|_{L^2_T(L^1_\xi)}\nonumber\\
          &\leq&  \left\|\int^t_0 \|1_A\xi_1M_{12}(t-s,\xi)\|_{L^1_{\xi_2}(L^2_{\xi_1})}
          \|\widehat{G}_2(\xi,s)\|_{L^\infty_{\xi_2}(L^2_{\xi_1})}ds\right\|_{L^2_T}\nonumber\\
           &\leq& C \left\|\int^t_0 (t-s)^{-\frac{3}{4}}
          \|\widehat{G}_2(\xi,s)\|_{L^\infty_{\xi_2}(L^2_{\xi_1})}ds\right\|_{L^2_T}\nonumber\\
          &\leq&C\left\|\widehat{G}_2(\xi,s)\right\|_{L^\frac{4}{3}_T(L^\infty_{\xi_2}(L^2_{\xi_1}))}\leq CA_{1,T}^2.\label{3.21}
        \end{eqnarray}
    Combining (\ref{3.14})-(\ref{3.15}) and (\ref{3.20})-(\ref{3.21}), we get
        \begin{equation}
    \left\|\widehat{\partial_1\psi}1_{\{|\xi|< 1,|\xi_1|>|\xi|^2\}}
    \right\|_{L^2_T(L^1_\xi)}  \leq CA_{2,0}+ C  A_{T}^2.
    \end{equation}
Similarly, we can estimate $ \left\|\widehat{v}1_{\{|\xi|< 1,|\xi_1|>|\xi|^2\}}
    \right\|_{L^2_T(L^1_\xi)}$ and omit the details.
\end{proof}

From Lemmas \ref{L3.1} and \ref{L3.2}, we can immediately obtain the following lemma.

 \begin{lem}
 Under the conditions in Lemma \ref{m0-E3.3-00}, then there holds
        \begin{equation}
  A_{2,T} \leq CA_{2,0}+CA_{1,T}+ C  A_{T}^2,\label{3.23}
    \end{equation}
    where   $C$ is a positive constant  independent of $T$.
\end{lem}

\section{Proof of Theorem \ref{m0-Thm1.1}}\label{m0-S5}

Via the analysis in \cite{Majda84}, we can get the following local existence
result following now the standard argument.
\begin{thm}\label{m0-Thm3.1}
Assume that the initial data $(\psi_0,v_0)$ satisfy (\ref{m0-E1.3})-(\ref{m0-E1.3-09}),   then there exists
     $T_0>0$ such that
     the system (\ref{m0-E1.5-N2}) has a unique local solution $(\psi,v,\nabla p)\in E^2_{T_0}$  on $[0,T_0]$.
\end{thm}

\noindent\textbf{Proof of Theorem \ref{m0-Thm1.1}.}
Theorem \ref{m0-Thm3.1} implies that the system (\ref{m0-E1.5-N2}) has a unique local strong solution $(\psi,v,\nabla p)$ on $[0,T^*)$, where $[0,T^{*})$ is the maximal existence interval of the above solution. The goal of this section is to prove that $T^*=\infty$ provided that the initial data $(\psi_0,v_0)$ satisfy (\ref{m0-E1.8}).

Assume that $(\psi,v,\nabla p)$ is the unique local strong solution of (\ref{m0-E1.5-N2}) on $[0,T^*)$, and satisfies $(\psi,v,\nabla p)\in E^2_T$, for all $T\in(0,T^*)$.
From (\ref{m0-E3.17}) and (\ref{3.23}), we have
    \begin{equation}
     A_{T}^2\leq
     C A_0^2
               + C  A_T^3(1+ A_T^2),
    \end{equation}
for all $T\in (0,T^*)$.
If  the initial data $(\psi_0,v_0)$ satisfy (\ref{m0-E1.8}), where $c_0$ satisfies
\begin{equation}
  C  \sqrt{2C }c_0(1+ 2C c_0^2)\leq \frac{1}{2},
\end{equation}
one can easily obtain
    \begin{equation}
     A_T^2\leq 2CA_0^2, \ \textrm{ for all }T\in (0,T^*).
    \end{equation}
Thus, we have that $T^*=\infty$,  and hence (\ref{m0-E1.12}) holds. From (\ref{m0-E3.14}), we see that $\nabla p\in C([0,\infty);H^1)$ and (\ref{m0-E1.13}) holds.
This finish proof of   Theorem \ref{m0-Thm1.1}.
{\hfill
$\square$\medskip}

\section*{Acknowledgements}
The author thanks Professor  Fanghua Lin   for his suggesting the problem, helpful instructions and valuable discussions.
The author thanks Professor Zhen Lei, Dong Li and  Yakov  Sinai  for their valuable discussions. Part of this work  was done when the  author
was visiting Department of Mathematics, Princeton University (2012),  Courant Institute of Mathematical
Sciences, New York University (2012), and The Institute of Mathematical Sciences, The Chinese University of Hong Kong (2014).  The   author appreciate the hospitality
from Professor Fanghua Lin, Yakov  Sinai and Zhouping Xin. This work is
  partially supported by  NSF of
China under Grants 11271322,  11331005 and 11271017, National Program for
Special Support of Top-Notch Young Professionals, Program for New Century
Excellent Talents in University NCET-11-0462, the Fundamental Research
Funds for the Central Universities (2012QNA3001).

\end{document}